\documentclass[a4paper,12pt]{amsart}

\usepackage[top=3cm,left=2.5cm,right=2.5cm,bottom=3cm]{geometry}
\usepackage{relsize}
\usepackage{lineno}
\usepackage{hyperref}
\usepackage{amsmath,amsthm,pb-diagram,amssymb,comment}
\usepackage{amsfonts,graphicx,color}
\usepackage{enumitem, fancyhdr, dsfont}
\usepackage{thmtools,cleveref}
\usepackage[normalem]{ulem}
\usepackage{stmaryrd}
\usepackage{textcomp}
\usepackage{esvect}
\usepackage{contour}
\usepackage{mathbbol}
\usepackage{fontawesome5}
\usepackage{tikz,tikz-cd,float}
\usepackage{chemformula}
\usepackage{amsthm,thmtools,xcolor} %put color in the theorems
\newlist{steps}{enumerate}{1}
\setlist[steps, 1]{label = Step \arabic*:}
\usepackage{multicol}
\hypersetup{
    colorlinks=true,
    linkcolor=dodger,
    filecolor=dodger,
    urlcolor=dodger,
    citecolor=dodger,
}

\setlist{topsep=0ex,itemsep=1ex}

\newcommand{\ran}{\mathrm{ran}}

\newcommand{\thzfc}{\mathsf{ZFC}}

\newcommand{\Mbf}{\mathbf{M}}
\newcommand{\Lbf}{\mathbf{L}}

\newcommand{\Pwf}{\mathcal{P}}

\newcommand{\bfrak}{\mathfrak{b}}

\newcommand{\dfrak}{\mathfrak{d}}

\newcommand{\Ior}{\mathds{I}}

\newcommand{\la}{\langle}
\newcommand{\ra}{\rangle}

\newcommand{\Dbf}{\mathbf{D}}

\newcommand{\Rbf}{\mathbf{R}}

\newcommand{\col}{\mathsf{col}}
\newcommand{\id}{\mathrm{id}}

\newcommand{\leqT}{\preceq_{\mathrm{T}}}
\newcommand{\eqT}{\cong_{\mathrm{T}}}

\renewcommand{\colon}{\nobreak\mskip2mu\mathpunct{}\nonscript
  \mkern-\thinmuskip{:}\allowbreak\mskip6muplus1mu\relax}%added \allowbreak
\newcommand{\set}[2]{\{#1\colon\,#2\}}

\newcommand{\seq}[2]{\la #1\colon\,#2\ra}

\contourlength{0.8pt}
\newcommand{\bsp}{\allowbreak\ }

\contourlength{0.8pt}

\newcount\skewfactor
\def\mathunderaccent#1#2 {\let\theaccent#1\skewfactor#2
\mathpalette\putaccentunder}
\def\putaccentunder#1#2{\oalign{$#1#2$\crcr\hidewidth
\vbox to.2ex{\hbox{$#1\skew\skewfactor\theaccent{}$}\vss}\hidewidth}}

\definecolor{ao(english)}{rgb}{0.0, 0.5, 0.0}
	\definecolor{ultramarineblue}{rgb}{0.25, 0.4, 0.96}
\definecolor{cornellred}{rgb}{0.7, 0.11, 0.11}
\definecolor{cobalt}{rgb}{0.0, 0.28, 0.67}
\definecolor{bleudefrance}{rgb}{0.19, 0.55, 0.91}
\definecolor{darkblue}{rgb}{0.0, 0.0, 0.55}
\definecolor{ferrarired}{rgb}{1.0, 0.11, 0.0}
\definecolor{brandeisblue}{rgb}{0.0, 0.44, 1.0}
\definecolor{azure(colorwheel)}{rgb}{0.0, 0.5, 1.0}
\definecolor{aqua}{rgb}{0.0, 1.0, 1.0}
\definecolor{aguamarina}{cmyk}{0.85,0,0.33,0}
\definecolor{cafe}{cmyk}{0,0.81,1,0.60}
\definecolor{canela}{cmyk}{0.14,0.42,0.56,0}
\definecolor{darkgray}{cmyk}{0,0,0,0.50}
\definecolor{emerald}{cmyk}{0.91,0,0.88,0.12}
\definecolor{fresa}{cmyk}{0,1,0.50,0}
\definecolor{gold}{cmyk}{0,0.10,0.84,0}
\definecolor{lightgray}{cmyk}{0,0,0,0.30}
\definecolor{marron}{cmyk}{0,0.72,1,0.45}
\definecolor{melon}{cmyk}{0,0.29,0.84,0}
\definecolor{ladri}{cmyk}{0,0.77,0.87,0}
\definecolor{olive}{cmyk}{0.64,0,0.95,0.40}
\definecolor{orange}{cmyk}{0,0.42,1,0}
\definecolor{peach}{cmyk}{0,0.46,0.50,0}
\definecolor{pink}{cmyk}{0,0.10,0.10,0}
\definecolor{orange}{cmyk}{0,0.42,1,0}
\definecolor{pine}{cmyk}{0.92,0,0.59,0.25}
\definecolor{purple}{cmyk}{0.45,0.86,0,0}
\definecolor{violet}{cmyk}{0.07,0.90,0,0.34}
\definecolor{craneorange}{RGB}{252,187,6}
\definecolor{red(ncs)}{rgb}{0.77, 0.01, 0.2}
\definecolor{americanrose}{rgb}{1.0, 0.01, 0.24}
\definecolor{hanpurple}{rgb}{0.32, 0.09, 0.98}

\definecolor{aguamarina}{cmyk}{0.85,0,0.33,0}
\definecolor{cafe}{cmyk}{0,0.81,1,0.60}
\definecolor{canela}{cmyk}{0.14,0.42,0.56,0}
\definecolor{darkgray}{cmyk}{0,0,0,0.50}
\definecolor{emerald}{cmyk}{0.91,0,0.88,0.12}
\definecolor{fresa}{cmyk}{0,1,0.50,0}
\definecolor{gold}{cmyk}{0,0.10,0.84,0}
\definecolor{lightgray}{cmyk}{0,0,0,0.30}
\definecolor{marron}{cmyk}{0,0.72,1,0.45}
\definecolor{melon}{cmyk}{0,0.29,0.84,0}
\definecolor{ladri}{cmyk}{0,0.77,0.87,0}
\definecolor{olive}{cmyk}{0.64,0,0.95,0.40}
\definecolor{orange}{cmyk}{0,0.42,1,0}
\definecolor{peach}{cmyk}{0,0.46,0.50,0}
\definecolor{pink}{cmyk}{0,0.10,0.10,0}
\definecolor{orange}{cmyk}{0,0.42,1,0}
\definecolor{pine}{cmyk}{0.92,0,0.59,0.25}
\definecolor{purple}{cmyk}{0.45,0.86,0,0}
\definecolor{violet}{cmyk}{0.07,0.90,0,0.34}

\DeclareSymbolFont{extraup}{U}{zavm}{m}{n}
\DeclareMathSymbol{\varheart}{\mathalpha}{extraup}{86}
\DeclareMathSymbol{\vardiamond}{\mathalpha}{extraup}{87}

\definecolor{dodger}{rgb}{0.0,0.5,1.0}

\definecolor{amber}{rgb}{1.0,0.49,0.0}

\definecolor{ogreen}{RGB}{107,142,35}

\title[Revisiting $\bfrak$ and $\dfrak$ through Interval Structures]
{ Revisiting $\bfrak$ and $\dfrak$ through Interval Structures}

\author{Miguel A. Cardona}
\address{Faculty of Engineering, Instituci\'on Universitaria Pascual Bravo. Calle 73 No.~73A -- 226, Medell\'in, Colombia;
and\newline 
Einstein Institute of Mathematics, 
Edmond J. Safra Campus, Givat Ram\\
The Hebrew University of Jerusalem\\
Jerusalem, 91904, Israel}
\email{\href{mailto:miguel.cardona@pascualbravo.edu.co}{miguel.cardona@pascualbravo.edu.co}}
\urladdr{\url{https://sites.google.com/view/miacardonamo}}

\author{Adam Marton}
 \address{Department of Applied Mathematics and Business Informatics, Faculty of Economics, Technical University of Košice, B. Němcovej 32, 040 01 Košice, Slovakia}
\email{\href{adam.marton@tuke.sk}{adam.marton@tuke.sk}}

\thanks{The first would like to thank the grant No.~PCT160, Direcci\'on de Tecnolog\'ia e Investigaci\'on and Oficina de Internacionalizaci\'on, Instituci\'on Universitaria Pascual Bravo}

\subjclass[2020]{03E05, 03E15, 03E17, 03E35, 03E40}

\keywords{bounding number, dominating number, interval partition,  infinite partition}

\definecolor{sub0}{RGB}{29,32,137}
\definecolor{sub1}{RGB}{1,71,157}
\definecolor{sub2}{RGB}{1,104,183}
\definecolor{sub3}{RGB}{0,160,234}
\definecolor{sug}{RGB}{0,154,68}
\definecolor{suy}{RGB}{208,219,1}

\DeclareUnicodeCharacter{0301}{\'{e}}

\begin{document}

\makeatletter
\def\@roman#1{\romannumeral #1}
\makeatother

\newcounter{enuAlph}
\renewcommand{\theenuAlph}{\Alph{enuAlph}}

\numberwithin{equation}{section}
\renewcommand{\theequation}{\thesection.\arabic{equation}}

% \numberwithin{figure}{section}
% \renewcommand{\thefigure}{\thesection.\arabic{figure}}
% \numberwithin{table}{section}
% \theoremstyle{plain}
%   \newtheorem{theorem}[equation]{Theorem}
%   \newtheorem{corollary}[equation]{Corollary}
%   \newtheorem{lemma}[equation]{Lemma}
%   \newtheorem{mainlemma}[equation]{Main Lemma}
%   \newtheorem{fact}[equation]{Fact}
%   \newtheorem{prop}[equation]{lemma}
%   \newtheorem{claim}[equation]{Claim}
%   \newtheorem{question}[equation]{QUESTION}
%   \newtheorem{problem}[equation]{Problem}
%   \newtheorem{conjecture}[equation]{Conjecture}
%   \newtheorem*{theorem*}{Theorem}
%   \newtheorem*{mainthm*}{Main Theorem}
%   \newtheorem{teorema}[enuAlph]{Theorem}
%   \newtheorem*{corollary*}{Corollary}
% %\theoremstyle{colored}{definition}
%   \newtheorem{definition}[equation]{Definition}
%   \newtheorem{example}[equation]{Example}
%   \newtheorem{remark}[equation]{Remark}
%   \newtheorem{notation}[equation]{Notation}
%   \newtheorem{context}[equation]{Context}
%   \newtheorem{observation}[equation]{Observation}
%   \newtheorem{assumption}[equation]{Assumption}
%   \newtheorem*{definition*}{Definition}
%   \newtheorem*{acknowledgements*}{Acknowledgements}

\theoremstyle{plain}
  \newtheorem{theorem}[equation]{Theorem}%[section]
  \newtheorem{corollary}[equation]{Corollary}
  \newtheorem{lemma}[equation]{Lemma}
  \newtheorem{mainlemma}[equation]{Main Lemma}
  \newtheorem*{mainthm}{Main Theorem}
  \newtheorem{prop}[equation]{lemma}
  \newtheorem{clm}[equation]{Claim}
  \newtheorem{subclm}[equation]{Subclaim}
  \newtheorem{fact}[equation]{Fact}
  \newtheorem{exer}[equation]{Exercise}
  \newtheorem{question}[equation]{Question}
  \newtheorem{problem}[equation]{Problem}
  \newtheorem{conjecture}[equation]{Conjecture}
  \newtheorem{assumption}[equation]{Assumption}
    \newtheorem{hopethm}[equation]{Hopeful Theorem}
    \newtheorem{hopelem}[equation]{Hopeful Lemma}
    \newtheorem{challenging}[enuAlph]{Main challenging}
    \newtheorem{hopele}[equation]{Hopeful Lemma}
    \newtheorem{discussion}[equation]{Discussion}
  \newtheorem*{thm}{Theorem}
  \newtheorem{teorema}[enuAlph]{Theorem}
  \newtheorem*{corolario}{Corollary}
\theoremstyle{definition}
  \newtheorem{definition}[equation]{Definition}
  \newtheorem{example}[equation]{Example}
  \newtheorem{remark}[equation]{Remark}
  \newtheorem{notation}[equation]{Notation}
  \newtheorem{context}[equation]{Context}

  \newtheorem*{defi}{Definition}
  \newtheorem*{acknowledgements}{Acknowledgements}

\def\sectionautorefname{Section}
\def\subsectionautorefname{Subsection}

\begin{abstract}
We investigate a family of relational systems arising from interval partitions of\/ $\omega$, inspired by Vojtáš's characterization of the bounding and dominating numbers. By varying the underlying asymptotic quantifiers and interval constraints, we obtain several natural interval-type generalizations.

We show that the universal variants are remarkably robust: in all the discrete, colored, restricted, bounded, and measure-theoretic settings considered here, the associated bounding and dominating numbers coincide with the classical invariants $\bfrak$ and $\dfrak$. In contrast, the existential variants systematically reverse these invariants, yielding bounding number $\dfrak$ and dominating number $\bfrak$.
\end{abstract}

\maketitle

%\tableofcontents

\section{Introduction and preliminaries}\label{s0}

Cardinal characteristics of the continuum measure the complexity of combinatorial and structural properties of the real line and related spaces. Classical examples include the bounding and dominating numbers, $\bfrak$ and $\dfrak$, which play a central role in Cichoń’s diagram and in the interaction between set theory, topology, and measure and category.

In his seminal work \cite{vojtasTM}, Vojtáš provided a precise characterization of the classical bounding and dominating numbers in terms of partitions of $\omega$ into finite intervals induced by strictly increasing functions. More concretely, if $\omega^{\uparrow\omega}$ denotes the set of strictly increasing functions from $\omega$ to $\omega$, then the bounding and dominating numbers admit the following equivalent formulations:
% \adam{I deleted the ``:'' from the definitions so as to keep the notation consistent with the rest of the paper. Feel free to change it back, if you think it is better for readability.}
\begin{align*}
\bfrak&= \min \bigl\{ |\mathcal{F}| : \mathcal{F} \subseteq \omega^{\uparrow\omega},\; \forall X \in [\omega]^\omega \; \exists f \in \mathcal{F} \ \neg(\forall^\infty n\;\bigl|\,[\,f(n),f(n+1))\cap X\bigr|\le 1) \bigr\}, \\
\dfrak&= \min \bigl\{ |\mathcal{S}| : \mathcal{S} \subseteq [\omega]^\omega,\; \forall f \in \omega^{\uparrow\omega} \; \exists X \in \mathcal{S} \ \forall^\infty n\;\bigl|\,[\,f(n),f(n+1))\cap X\bigr|\le 1 \bigr\}.
\end{align*}
This viewpoint not only clarifies the combinatorial nature of these invariants but also opens the door to systematic generalizations by modifying the underlying relations. 

The framework of relational systems has since proven to be a powerful and flexible tool for studying cardinal characteristics (see~e.g.,\cite{Vojtas,vojtasamoeba,blass}. A \textit{relational system}  $\Rbf=\la X, Y, \sqsubset\ra$  consists of two non-empty sets $X$ and $Y$, together with a relation $\sqsubset$.
\begin{enumerate}[label=(\arabic*)]
    \item A set $F\subseteq X$ is \emph{$\Rbf$-bounded} if $\exists\, y\in Y\ \forall\, x\in F\ x \sqsubset y$. 
    \item A set $E\subseteq Y$ is \emph{$\Rbf$-dominating} if $\forall\, x\in X\ \exists\, y\in E\ x \sqsubset y$. 
\end{enumerate}
We associate two cardinal characteristics with the relational system $\Rbf$:
\begin{itemize}
    \item[{}] $\bfrak(\Rbf):=\min\{|F|\colon  F\subseteq X  \text{ is }\Rbf\text{-unbounded}\}$ the \emph{bounding number of $\Rbf$}, and
    
    \item[{}] $\dfrak(\Rbf):=\min\{|D|\colon  D\subseteq Y \text{ is } \Rbf\text{-dominating}\}$ the \emph{dominating number of $\Rbf$}.
\end{itemize}

The interval characterization introduced by Vojt\'a\v{s} fits naturally into the general framework of relational systems described above. From this perspective, the bounding and dominating numbers arise from a simple asymptotic relation measuring how infinite subsets of $\omega$ distribute inside the interval partitions generated by strictly increasing functions. More importantly, this viewpoint suggests that the specific relation considered by Vojt\'a\v{s} is not isolated, but rather part of a much broader interval paradigm. 

The main objective of this paper is to show that this phenomenon is remarkably robust: after varying the underlying asymptotic condition in several natural ways, combinatorial, geometric, colored, bounded, and measure-theoretic, the resulting relational systems continue to recover the classical invariants $\bfrak$ and $\dfrak$, often exhibiting unexpected duality phenomena between their universal and existential forms.

We introduce a hierarchy of interval relations parameterized by an integer $k\ge 1$. For $\iota\in\{\forall,\exists\}$, let $f\sqsubset_\iota^k X$ express that the intervals $[f(n),f(n+1))$ contain at most $k$ points of $X$ either eventually ($\iota=\forall$) or infinitely often ($\iota=\exists$). This yields the relational systems $\Rbf_\iota^k
=
\langle
\omega^{\uparrow\omega},
[\omega]^\omega,
\sqsubset_\iota^k
\rangle.$ Our first result reveals a sharp dichotomy between the universal and existential regimes.

\begin{teorema}[\autoref{lem:forall}, \autoref{a32}--\autoref{a61}]\label{mthmone}
For every $k\ge1$, $\bfrak(\Rbf_\forall^k)=\bfrak$, $\dfrak(\Rbf_\forall^k)=\dfrak$, whereas $\bfrak(\Rbf_\exists^k)=\dfrak$, $\dfrak(\Rbf_\exists^k)=\bfrak.$
\end{teorema}

Thus, merely replacing ``eventually'' by ``infinitely often'' completely reverses the associated cardinal invariants. In particular, the interval paradigm naturally encodes both sides of the classical $(\bfrak,\dfrak)$-duality.

We show that this phenomenon persists under several substantial modifications of the underlying combinatorics. First, instead of counting points of a set inside an interval, we count the number of blocks of an infinite partition intersected by the interval. This leads to the colored systems $\Rbf_\iota^{\col,k}=
\langle
\omega^{\uparrow\omega},
\text{Part}(\omega),
\sqsubset_\iota^{\col,k}
\rangle$ for $\iota\in\{\forall,\exists\}$
where $\text{Part}(\omega)$ denotes the family of infinite partitions of $\omega$ into finite sets.

\begin{teorema}[\autoref{a65}--\autoref{a66}]\label{mthmthird}
For every $k\ge2$, $\bfrak(\Rbf_\forall^{\col,k})=\bfrak$, $\dfrak(\Rbf_\forall^{\col,k})=\dfrak$, and $\bfrak(\Rbf_\exists^{\col,k})=\dfrak$, $\dfrak(\Rbf_\exists^{\col,k})=\bfrak$.
\end{teorema}

Next, we impose geometric restrictions on the interval decomposition itself. Let
\[
\Dbf_k
=
\{
f\in\omega^{\uparrow\omega}
:
|[f(n),f(n+1))|>k
\text{ for all }n
\},
\]
and consider
\[
\Lbf_\iota^k
=
\langle
\Dbf_k,
[\omega]^*,
\sqsubset_\iota^k
\rangle,
\]
where $[\omega]^*$ is the family of infinite co-infinite subsets of $\omega$.

\begin{teorema}[\autoref{a80}--\autoref{a81}]\label{mthmfive}
For every $k>0$, $\bfrak(\Lbf_\forall^k)=\bfrak$, $\dfrak(\Lbf_\forall^k)=\dfrak$ while $\bfrak(\Lbf_\exists^k)=\dfrak$, $\dfrak(\Lbf_\exists^k)=\bfrak$ for any $k<\omega$. 
\end{teorema}

The interval method also admits a measure-theoretic counterpart. Replacing finite counting by Lebesgue measure, we define relational systems $\Mbf_\iota^\varepsilon$ for $\varepsilon>0$
by requiring the intersections
$
[f(n),f(n+1))\cap Y
$
to have measure at most $\varepsilon$ eventually or infinitely often. Despite their analytic appearance, these systems are Tukey equivalent to the original combinatorial relations.

\begin{teorema}[\autoref{a83}--\autoref{a116}]\label{mthmsix}
For every $\varepsilon>0$ and $k>0$, $\Mbf_\iota^\varepsilon\eqT\Rbf_\iota^k$ for $\iota\in\{\forall,\exists\}$. Consequently, $\bfrak(\Mbf_\forall^\varepsilon)=\bfrak$, $\dfrak(\Mbf_\forall^\varepsilon)=\dfrak$ and $\bfrak(\Mbf_\exists^\varepsilon)=\dfrak$, $\dfrak(\Mbf_\exists^\varepsilon)=\bfrak$.
\end{teorema}

We also investigate interval relations with variable asymptotic tolerances. 
Instead of requiring a fixed bound $k$, we allow the admissible number of points to increase with the interval index itself. Let
\[
\omega^{>\id}
=
\{
f\in\omega^{\uparrow\omega}
:
f(n+1)-f(n)>n
\},
\]
and for $\iota\in\{\forall,\exists\}$ define
\[
\Rbf^{\id}_\iota
=
\langle
\omega^{>\id},
[\omega]^\omega,
\sqsubset_{\id}
\rangle,
\]
where
\[
f\sqsubset^{\id}_\iota X
\; \iff\;
\iota^\infty n\
\bigl|[f(n),f(n+1))\cap X\bigr|
\le n.
\]

At first sight, allowing the bounds to diverge suggests a substantially weaker asymptotic restriction. Surprisingly, the associated cardinal characteristics remain unchanged.

\begin{teorema}[\autoref{a120}-\ref{a121}]\label{mthmseven}
$\bfrak(\Rbf^{\id}_\forall)
=
\bfrak$ and $\dfrak(\Rbf^{\id}_\forall)
=
\dfrak$ while $\bfrak(\Rbf^{\id}_\exists)
=
\bfrak$ and $\dfrak(\Rbf^{\id}_\exists)
=
\dfrak$
\end{teorema}

Finally, we study bounded interval relations in which the parameter $k$ is no longer fixed in advance, but may depend on the witnessing pair itself. Let
\[
\omega^\omega_\uparrow
=
\{
f\in\omega^\omega
:
f(n+1)-f(n)\to\infty
\},
\]
and define
\[
\Rbf_\iota^{\mathrm{bd}}
=
\langle
\omega^\omega_\uparrow,
[\omega]^\omega,
\sqsubset_\iota^{\mathrm{bd}}
\rangle,
\qquad
\iota\in\{\forall,\exists\},
\]
where
\[
f\sqsubset_\forall^{\mathrm{bd}}X
\; \iff\;
\exists k\in\omega\
\forall^\infty n\
\bigl|[f(n),f(n+1))\cap X\bigr|
\le k,
\]
and
\[
f\sqsubset_\exists^{\mathrm{bd}}X
\; \iff\;
\exists k\in\omega\
\exists^\infty n\
\bigl|[f(n),f(n+1))\cap X\bigr|
\le k.
\]

These bounded variants show that the universal/existential dichotomy persists even when the threshold is allowed to vary dynamically.

\begin{teorema}[\autoref{a84}--\autoref{a101}]\label{mthmtwo}
In $\thzfc$, $\bfrak(\Rbf_\exists^{\mathrm{bd}})
=
\dfrak$, $\dfrak(\Rbf_\exists^{\mathrm{bd}})
=
\bfrak$ and $\dfrak(\Rbf_\forall^{\mathrm{bd}})
=
\dfrak$, $\mathfrak v(\Rbf_\forall^{\mathrm{bd}})
=
\bfrak$. 
\end{teorema}

Together, these results show that the classical pair $(\bfrak,\dfrak)$ is encoded in a remarkably broad spectrum of interval phenomena. The persistence of the same invariants across fixed, variable, colored, geometric, bounded, and measure-theoretic settings suggests that Vojt\'a\v{s}-type interval systems form a robust universality class governing asymptotic sparsity on $\omega$.

Methodologically, the proofs throughout the paper rely heavily on Tukey connections between relational systems.
This approach allows us to dualize cardinal inequalities (for a general overview
on this framework, see, e.g.,~\cite{Vojtas,blass}). 

Let $\Rbf=\la X,Y,{\sqsubset}\ra$ and $\Rbf'=\la X',Y',{\sqsubset'}\ra$ be two relational systems. We say that
$(\Psi_-,\Psi_+)\colon\Rbf\to\Rbf'$
is a~\emph{Tukey connection from $\Rbf$ into $\Rbf'$} if
$\Psi_-\colon X\to X'$ and $\Psi_+\colon Y'\to Y$ are functions such that
\[
\forall x\in X\;\forall y'\in Y'\bsp
(\Psi_-(x)\sqsubset' y' \Rightarrow x\sqsubset \Psi_+(y')).
\]
The \emph{Tukey order} between relational systems is defined by
$\Rbf\leqT\Rbf'$ iff there is a~Tukey connection from $\Rbf$ into $\Rbf'$. \emph{Tukey equivalence} is defined by $\Rbf\eqT\Rbf'$ iff $\Rbf\leqT\Rbf'$ and $\Rbf'\leqT\Rbf$.

The dual of $\Rbf$ is defined by $\Rbf^\perp:=\la Y,X,{\sqsubset^\perp}\ra$ where $y\sqsubset^\perp x$ iff $x\not\sqsubset y$. Note that $\bfrak(\Rbf^\perp)=\dfrak(\Rbf)$ and $\dfrak(\Rbf^\perp)=\bfrak(\Rbf)$.

\begin{example}\label{b-d}
Let $\Dbf=\la\omega^{\uparrow\omega},\omega^{\uparrow\omega},\leq^*\ra$ where $x\leq^* y$ means $\forall^\infty n\; x(n)\leq y(n)$. Then $\bfrak:=\bfrak(\Dbf)$ and $\dfrak:=\dfrak(\Dbf)$.
\end{example}

\section{ZFC-results}

In this section, we establish all the main results stated in \autoref{s0}. Throughout, we analyze the interval relational systems introduced above and compute their associated bounding and dominating numbers within $\thzfc$.

\begin{definition}\label{def:bas}
Let $k\ge1$.
    \begin{enumerate}[label=\rm(\alph*)]
        \item For $f \in \omega^{\uparrow\omega}$ and $X \in [\omega]^\omega$, let \[f\;\sqsubset_\forall^{k}\;X\; \iff\;\forall^\infty n\;\bigl|\,[\,f(n),f(n+1))\cap X\bigr|\le k,\]
\[f\;\sqsubset_\exists^{k}\;X\; \iff\;\exists^\infty n\;\bigl|\,[\,f(n),f(n+1))\cap X\bigr|\le k.\]
        \item Define the relational system $\Rbf_\iota^k=\la\omega^{\uparrow\omega},[\omega]^\omega,\sqsubset_\iota^k\ra$ for $\iota\in\{\forall,\exists\}$. 
    \end{enumerate}
\end{definition}

We begin by distinguishing the case $k=0$ for $\sqsubset_\forall^{k}$.

\begin{remark}
 It is easy to see that any infinite set $X$ has a nonempty intersection with infinitely many intervals of any interval partition, therefore, $f\not\sqsubset^0_\forall X$ for every $f\in \omega^{\uparrow\omega}$, $X\in [\omega]^\omega$ and consequently $\mathfrak{b}(\mathbf{R}_\forall^0) = 1$ and $\mathfrak{d}(\mathbf{R}_\forall^0)$ is not well-defined or $\infty$.
\end{remark}

We now record the basic monotonicity properties of these relational systems. 

\begin{fact}\label{fac:bas}
Let $f\in \omega^{\uparrow\omega}$ and $X\in [\omega]^\omega$. 
    \begin{enumerate}[label=\rm(\arabic*)]
        \item\label{fac:basc}  If $k<\ell$ then $\Rbf^\ell_\iota\leqT\Rbf^k_\iota$ for $\iota\in\{\forall,\exists\}$. 
         \item\label{fac:basb} For each $k\in\omega$, we have $\Rbf^k_\exists\leqT\Rbf^k_\forall$.
    \end{enumerate}
\end{fact}

As recalled in \autoref{s0}, Vojt\'a\v{s}~\cite{vojtasTM} showed that $\bfrak(\Rbf_\forall^{1})=\bfrak$ and $\dfrak(\Rbf_\forall^{1})=\dfrak$, so by \autoref{fac:bas} we have $\bfrak\leq\bfrak(\Rbf_\forall^{k})$ and  $\dfrak\geq\dfrak(\Rbf_\forall^{k})$. Therefore, in order to establish the converse, it suffices to prove the following lemma.

\begin{lemma}\label{lem:forall}
For $k\geq2$, we have $\Dbf\leqT\Rbf_\forall^{k}$. In particular, $\bfrak\geq\bfrak(\Rbf_\forall^{k})$ and  $\dfrak\leq\dfrak(\Rbf_\forall^{k})$. 
\end{lemma}

\begin{proof}
Let $\Psi_-\colon \omega^{\uparrow\omega}\to\omega^{\uparrow\omega}$ be defined as the identity map. To define $\Psi_+\colon[\omega]^\omega\to\omega^{\uparrow\omega}$, first define $h_X\colon \omega \to \omega$ for every $X\in[\omega]^\omega$ as follows. Let $X = \{x_0, x_1, x_2, \ldots\}$ be the increasing enumeration of $X$, i.e.,  $x_0 < x_1 < x_2 < \cdots$.

Define $h_X\colon \omega \to \omega$ by:
\[
h_X(n) = x_{(k+1)n}.
\]
Put $\Psi_+(X)=h_X$. Now, let $f \in \omega^{\uparrow\omega}$, $X\in[\omega]^\omega$, and suppose that $f \sqsubset_\forall^{k} X$. This means there exists $N \in \omega$ such that for all $n \geq N$,
\[
\left| [f(n), f(n+1)) \cap X \right| \leq k.
\]
For each $n \in \omega$, let \[a_n=|\set{i\in\omega}{x_i< f(n)}|.\]
By the definition of $\seq{a_n}{n\in\omega}$, for every $n\geq 1$ we have
\[
x_{a_n-1} < f(n) \leq x_{a_n}.
\]

For $n \geq N$, $f \sqsubset_\forall^{k} X$ implies that the number of elements of $X$ in ${[f(n), f(n+1))}$ is at most $k$. Thus,
\[
a_{n+1} \leq a_n + k.
\]

Solving this recurrence for $n \geq N$,
\[
a_n \leq a_N + k(n - N) \leq c + kn,
\]
where $c = a_N - kN$ is a constant. Therefore, for all $n \geq N$,
\[
a_n \leq c + kn.
\]

Since $f(n) \leq x_{a_n}$ and the sequence $\seq{x_i}{i\in\omega}$ is increasing, for $n\geq N$ we have
\[
f(n) \leq x_{a_n} \leq x_{c + kn}.
\]

Now, consider $h_X(n) = x_{(k+1)n}$. For sufficiently large $n$ (specifically, when $n > c$),
\[
(k+1)n > c + kn,
\]
so
\[
x_{(k+1)n} > x_{c + kn}.
\]
Thus,  for $n>\max\{c, N\}$,
\[
f(n) \leq x_{c + kn} < x_{(k+1)n} = h_X(n),
\]
which implies $f(n) \leq h_X(n)$.

Therefore, $f(n) \leq h_X(n)$ for all but finitely many $n$, i.e.,  $f \leq^* h_X$.
\end{proof}

The existential variant behaves dually. More precisely, the next two lemmas identify the cardinal characteristics associated with $\Rbf_\exists^1$.

\begin{lemma}\label{a31}
$\Rbf_\exists^1 \leqT \Dbf^\perp$. In particuar, $\dfrak(\Rbf_\exists^1)\leq\bfrak$ and $\bfrak(\Rbf_\exists^1)\geq \dfrak$. 
\end{lemma}

\begin{proof}
Define the maps:
\[
\Phi_-\colon \omega^{\uparrow \omega} \to \omega^{\uparrow \omega}, \quad \Phi_-(f) = f,
\]
\[
\Phi_+\colon \omega^{\uparrow \omega} \to [\omega]^\omega, \quad \Phi_+(b) = b[\omega] := \{ b(n) : n \in \omega \}.
\]
We verify the Tukey condition: for any $f, b \in \omega^{\uparrow \omega}$,
\[
\Phi_-(f) (\leq^*)^\perp b\; \Rightarrow\; f \sqsubset_\exists^1 \Phi_+(b).
\]
Assume $b\not\leq^* \Phi_-(f) $, i.e., $b \not\leq^* f$. This means $\exists^\infty n\; \, b(n) > f(n)$. We prove that $f \sqsubset_\exists^1 b[\omega]$. Suppose, for contradiction, that $\neg (f \sqsubset_\exists^1 b[\omega])$, i.e., $\exists N \, \forall n \geq N \, |[f(n), f(n+1)) \cap b[\omega]| \geq 2$. For each $n \geq N$, let
\[
s(n) = \min \{ j : b(j) \in [f(n), f(n+1)) \}.
\]
Because each interval $[f(n), f(n+1))$ contains at least two points of $b[\omega]$, we have $s(n+1) \geq s(n) + 2$, so $s(n) \geq s(N) + 2(n - N)$. Consequently, there exists $n_0$ such that $s(n) > n$ for all $n > n_0$. For such $n$, since $n < s(n)$, the minimality of $s(n)$ implies $b(j) < f(n)$ for all $j < s(n)$; in particular, $b(n) < f(n)$. This yields $b(n) < f(n)$ for all sufficiently large $n$, i.e., $b \leq^* f$, which contradicts the assumption $b \not\leq^* f$. Hence, $f \sqsubset_\exists^1 b[\omega]$.
Thus, $\Rbf_\exists^1 \leqT \Dbf^\perp$.
\end{proof}

To obtain the reverse inequalities, we relate $\Rbf_\exists^1$ to the interval inclusion relation introduced by Blass \cite{blass}.

\begin{definition}
Define the following relation on $\omega^{\uparrow\omega}$:
   \[ f \sqsubseteq g \quad\Longleftrightarrow\quad \forall^\infty n\in\omega\; \exists m\in\omega\; [f(m), f(m+1))\subseteq [g(n), g(n+1)).
   \]
Let $\Ior = \langle \omega^{\uparrow\omega}, \omega^{\uparrow\omega}, \sqsubseteq \rangle$. Blass~\cite{blass} proved that $\bfrak(\Ior)=\bfrak$ and $\dfrak(\Ior)=\dfrak$.
 \end{definition}

\begin{lemma}\label{a30}
For every $k \geq 1$, $\Ior^\perp \leqT \Rbf_\exists^k$. 
Consequently, $\bfrak(\Rbf_\exists^k) \leq \dfrak$ and $\dfrak(\Rbf_\exists^k) \geq \bfrak$.    
\end{lemma}

\begin{proof}
 Define maps 
\[ \Psi_-\colon \omega^{\uparrow\omega} \to \omega^{\uparrow\omega} \quad \text{and} \quad \Psi_+\colon [\omega]^\omega \to \omega^{\uparrow\omega} \]
such that for every $g \in \omega^{\uparrow\omega}$ and $X \in [\omega]^\omega$:
\[ \Psi_-(g) \sqsubset_\exists^k X \;\Rightarrow\; g \sqsubseteq^\perp \Psi_+(X).\]
By the definition of the dual relation $\sqsubseteq^\perp$, the condition $g \sqsubseteq^\perp \Psi_+(X)$ is equivalent to $\Psi_+(X) \not\sqsubseteq g$. Thus, we need to satisfy the contrapositive:
\[ \Psi_+(X) \sqsubseteq g \;\Rightarrow\; \neg(\Psi_-(g) \sqsubset_\exists^k X).\]

Let $\Psi_-(g) = g$ (the identity function).   For $X \in [\omega]^\omega$, let $e_X$ be the increasing enumeration of $X$. Define $\Psi_+(X) = f_X \in \omega^{\uparrow\omega}$ by
    \[ f_X(n) = e_X(n \cdot (k+1)).\]

Suppose $\Psi_+(X) \sqsubseteq g$, which means $f_X \sqsubseteq g$. By the definition of $\sqsubseteq$ we have
\[ \forall^\infty n\; \exists m\; [f_X(m), f_X(m+1)) \subseteq [g(n), g(n+1)).\]
Consider such an interval $[f_X(m), f_X(m+1))$. By the definition of $f_X$ we have
\[ [f_X(m), f_X(m+1)) = [e_X(m(k+1)), e_X((m+1)(k+1))). \]
This interval contains the elements $\{e_X(m(k+1)), e_X(m(k+1)+1), \dots, e_X(m(k+1)+k)\}$. There are exactly $k+1$ elements of $X$ in this interval.

Since $[f_X(m), f_X(m+1)) \subseteq [g(n), g(n+1))$, it follows that
\[ |[g(n), g(n+1)) \cap X| \geq |[f_X(m), f_X(m+1)) \cap X| = k+1>k,\]
thus, we have $|[g(n), g(n+1)) \cap X| > k$ for all but finitely many $n$. 

By the definition of $\sqsubset_\exists^k$, the condition $g \sqsubset_\exists^k X$ requires infinitely many $n$ such that $|[g(n), g(n+1)) \cap X| \leq k$. However, we just showed that for almost all $n$, the intersection has size $> k$. 
Therefore, we have $\neg(g \sqsubset_\exists^k X)$, i.e.,  $\neg(\Psi_-(g) \sqsubset_\exists^k X)$.   
\end{proof}

Combining \autoref{fac:bas}\ref{fac:basc}, \autoref{a31}, and \autoref{a30}, we immediately obtain the following corollary.

\begin{corollary}\label{a32}
For any $k\geq1$, we have $\bfrak=\dfrak(\Rbf_\exists^{k})$ and $\dfrak=\bfrak(\Rbf_\exists^{k})$. 
\end{corollary}

We now treat the remaining case $k=0$.

\begin{lemma}\label{a60}
    $\Rbf^0_\exists \leqT\Rbf^1_\exists$.
\end{lemma}
\begin{proof}
Define $\Psi_-\colon\omega^{\uparrow\omega}\to\omega^{\uparrow\omega}$ by $\Psi_-(f)(n)=f(2n)$ (i.e., we glue together every pair of consecutive intervals), and put $\Psi_+=\id_{[\omega]^\omega}$. Assume $\exists^\infty n\;|[f(2n), f(2n+2))\cap X|\leq 1$. By the assumption there exist infinitely many pairs of intervals containing at most one point from $X$, therefore, one of the two intervals must have empty intersection with $X$, so $\exists^\infty n\; |[f(n), f(n+1))\cap X|=0$, i.e, $f\sqsubset^0_\exists X=\Psi_+(X)$.
\end{proof}

Together with \autoref{fac:bas}\ref{fac:basc}, \autoref{a32}, and \autoref{a60}, this yields the following consequence.

\begin{corollary}\label{a61}
    $\dfrak(\Rbf^0_\exists)=\bfrak$ and $\bfrak(\Rbf^0_\exists)=\dfrak$.
\end{corollary}

This completes the proof of \autoref{mthmone}. We next consider a partition variant of the previous relations.

 \begin{definition}
 Let $\text{Part}(\omega)$ be the set of all infinite partitions $\mathcal{P} = \{X_i : i \in \omega\}$ of $\omega$ into finite sets.
\begin{enumerate}[label=\rm(\alph*)]
    \item For $f\in\omega^{\uparrow\omega}$ and $\Pwf\in\text{Part}(\omega)$, let
    \[ f \sqsubset_{\forall}^{\col,k} \mathcal{P} \; \iff\; \forall^\infty n \in \omega \ \bigl| \{ i \in \omega : [f(n), f(n+1)) \cap X_i \neq \emptyset \} \bigr| \le k, \]  
\[ f \sqsubset_{\exists}^{\col,k} \mathcal{P} \; \iff\; \exists^\infty n \in \omega \ \bigl| \{ i \in \omega : [f(n), f(n+1)) \cap X_i \neq \emptyset \} \bigr| \le k. \] 
    \item Define the relational system $\Rbf_{\iota}^{\col,k} = \langle \omega^{\uparrow\omega}, \text{Part}(\omega), \sqsubset_{\iota}^{\col,k} \rangle$\footnote{``col'' stands for color} for $\iota\in\{\forall,\exists\}$. 
\end{enumerate}
\end{definition}

The case $k=1$ is degenerate for the universal version, as shown by the following remark.

\begin{remark}
We show that $\bfrak(\Rbf_{\forall}^{\col,1})=2$. Let $f,g$ be such that $f(n)<g(n)<f(n+1)<g(n+1)$ for all $n\in\omega$. Assume there is a partition $\Pwf$ such that $f,g\sqsubset^1_\text{col} \Pwf$. Then there are $n_1, n_2$ such that for all $n\geq \max\{n_1, n_2\}$ exist $P_{n,f},P_{n,g}\in\Pwf$ such that $[f(n), f(n+1))\subseteq P_{n,f}$ and $[g(n), g(n+1))\subseteq P_{n,g}$. Since intervals $[f(n), f(n+1))$ and $[g(n), g(n+1))$ overlap, we must have $P_{n,f}=P_{n,g}$. Since intervals $[g(n), g(n+1))$ and $[f(n+1), f(n+2))$ overlap, we must have $P_{n,f}=P_{n,g}=P_{n+1, f}$ and so on. So we have that each $n$-th interval for $n\geq \max\{n_1, n_2\}$  of each of two functions has the same color, which is a contradiction with $\Pwf\subseteq [\omega]^{<\omega}$.   
\end{remark}

\begin{fact}\label{fct:partbd}
\ 
\begin{enumerate}[label=\rm(\arabic*)]
    \item $\Rbf_{\iota}^{\col,\ell}\leqT\Rbf_{\iota}^{\col,k}$ for $k<\ell$ and $\iota\in\{\forall,\exists\}$.
    \item  $\Rbf_{\exists}^{\col,k}\leqT\Rbf_{\forall}^{\col,k}$ for each $k\in\omega$. 
\end{enumerate}    
\end{fact}

\begin{lemma}\label{lem1:partbd}
 $\Rbf_{\forall}^{\col,2} \leqT \Ior$. In particular, $\bfrak\leq \bfrak(\Rbf_{\forall}^{\col,2} )$ and $\dfrak\geq \dfrak(\Rbf_{\forall}^{\col,2} )$.
 \end{lemma}

\begin{proof}
     We construct a Tukey connection $(\Psi_-, \Psi_+)$, where
\[ \Psi_-\colon \omega^{\uparrow\omega} \to \omega^{\uparrow\omega} \quad \text{and} \quad \Psi_+\colon \omega^{\uparrow\omega} \to \text{Part}(\omega) \]
are such that
\[ \Psi_-(f) \sqsubseteq g \;\Rightarrow\; f \sqsubset_{\text{col}}^{2} \Psi_+(g). \]

Define
\begin{itemize}
    \item  $\Psi_-(f) = f$, and
    \item $\Psi_+(g) = \mathcal{P}_g$, where $\mathcal{P}_g = \{ [g(i), g(i+1)) : i < \omega \}$.
\end{itemize}
Assume $\Psi_-(f)=f\sqsubset g$, i.e., $\forall^\infty m\; \exists n_m\;[f(n_m), f(n_m+1))\subseteq [g(m), g(m+1))$. Let $M\in\omega$ be such that for all $m\geq M$ there is $n_m$ with $[f(n_m), f(n_m+1))\subseteq [g(m), g(m+1))$. Clearly, $n_m\geq n_M$ for all $m\geq M$. Fix any $n\geq n_M$. Then either $n=n_m$ for some $m\geq M$ (i.e. $[f(n), f(n+1))\subseteq [g(m), g(m+1))$ and consequently $[f(n), f(n+1))$ has nonempty intersection only with $[g(m), g(m+1))$) or $[f(n), f(n+1))\subset[g(m), g(m+1))\cup[g(m+1), g(m+2))$ for some $n\geq N$. It cannot happen that $[f(n),f(n+1))$ has a nonempty intersection with three (or more) intervals constituted of $g$ since then we should have $[g(m), g(m+1))\subsetneq [f(n), f(n+1))$ for some $m\geq M$ (the middle interval of the three mentioned) and that is a contradiction with the assumption that $[g(m), g(m+1))$ contains at least one interval $[f(j), f(j+1))$. Hence, all but finitely many $[f(n), f(n+1))$ have a nonempty intersection with at most two intervals constituted by $g$, i.e., $f\sqsubset^2_\text{col} \Pwf_g$.
\end{proof}

\begin{lemma}\label{lem2:partbd}
 For any $k\geq2$, we have $\Rbf_\iota^k\leqT \Rbf_{\iota}^{\col,k}$ for $\iota\in\{\forall,\exists\}$. In particular, $\bfrak(\Rbf_{\iota}^{\col,k})\leq\bfrak(\Rbf_\iota^k)$ and $\dfrak(\Rbf_\iota^k)\leq\dfrak(\Rbf_{\iota}^{\col,k})$.
\end{lemma}
\begin{proof}
    Let $\Psi_-=\id_{\omega^{\uparrow\omega}}$ and for $\Pwf=\set{P_n}{n\in\omega}\in\text{Part}(\omega)$ let \[\Psi_+(\Pwf)=\set{\min(P_n)}{n\in\omega}.\]
    Clearly, $\Psi_+(\Pwf)\in [\omega]^\omega$ for all $\Pwf\in \text{Part}(\omega)$. 

    It remains to show that \[f\sqsubset_{\iota}^{\col,k} \Pwf\;\Rightarrow\; f\sqsubset^k_\iota\Psi_+(\Pwf).\]
    We only present the proof for $\iota=\forall$, as the remaining case is analogous. Assume $f\sqsubset_{\forall}^{\col,k} \Pwf$, i.e., there is $N\in\omega$ such that  $\forall n\geq N\;\exists F_n\in [\omega]^{\leq k}\; [f(n), f(n+1))\subseteq \bigcup_{j\in F_n} P_j$. Fix $n\geq N$. Then it follows immediately from $[f(n), f(n+1))\subseteq \bigcup_{j\in F_n} P_j$ that $[f(n), f(n+1))$ contains at most $|F_n|\leq k$ minimums of elements of $\Pwf$. Therefore, $\forall^\infty n\;|[f(n), f(n+1))\cap \Psi_+(\Pwf)|\leq k$, i.e., $f\sqsubset^k_\forall \Psi_+(\Pwf)$.
\end{proof}

As an immediate consequence of \autoref{lem1:partbd} and \autoref{lem2:partbd}, we obtain the following result, establishing the first part of \autoref{mthmthird}.

 \begin{corollary}\label{a65}
 For any $k\geq2$, we have $\bfrak= \bfrak(\Rbf_{\forall}^{\col,2})=\bfrak(\Rbf_{\forall}^{\col,k})$ and $\dfrak= \dfrak(\Rbf_{\forall}^{\col,2} )=\dfrak(\Rbf_{\forall}^{\col,k} )$.   
\end{corollary}

For the existential counterpart, we first isolate a simple combinatorial observation.

\begin{fact}\label{Nintervals}
Let $I, I_0,\dotso, I_{N-1}$ be integer intervals such that $I_0,\dotso, I_{N-1}$ are mutually disjoint and $I\cap I_j\neq \emptyset$ for all $j<N$. Then $I$ contains at least $N-2$ intervals from $\{I_0,\dotso, I_{N-1}\}$ as subsets.
\end{fact}

\begin{lemma}\label{lem3:partbd}
    $\Rbf_{\exists}^{\col,2}\leqT\Ior^\perp$.
\end{lemma}

\begin{proof}
    Let $f, g \in \omega^{\uparrow\omega}$. We define $\Psi_-(f) = f$ and $\Psi_+(g)=\Pwf$, where $\mathcal{P} = \{A_i : i < \omega\}$ is defined as follows:
\begin{itemize}
    \item $A_0 = [0, g(1))$,
    \item $A_i = [g(i), g(i+1))$ for $i \ge 1$.
\end{itemize}
This forms a partition of $\omega$ into finite, disjoint sets.

Assume $\Psi_{-}(f)\not\sqsupseteq g$, that is, there exists an infinite set $X \subseteq \omega$ such that for all $n \in X$,
\[ \forall m \; [g(m), g(m+1)) \not\subseteq [f(n), f(n+1)).\]
We need to show that for each $n \in X$, the interval $I_n = [f(n), f(n+1))$ intersects at most two sets $A_i$. Let $n \in X$. Suppose for the sake of contradiction that $I_n$ intersects $N\geq 3$ partition sets $A_{i_1},\dotso,A_{i_N}$. Note that these sets are intervals by the definition. It follows from \autoref{Nintervals} that $[f(n), f(n+1))$ contain at least $N-2\geq 3-2=1$ interval, say $A_{i_j}$. Since $A_{i_j}=[g(i_j), g(i_j+1))$, we get the contradiction with the assumption $n\in X$.
\end{proof}

Combining \autoref{lem2:partbd} and \autoref{lem3:partbd}, we obtain the following corollary, which completes \autoref{mthmthird}.

 \begin{corollary}\label{a66}
 For any $k\geq2$, we have $\bfrak= \dfrak(\Rbf_{\exists}^{\col,2})=\dfrak(\Rbf_{\exists}^{\col,k})$ and $\dfrak= \bfrak(\Rbf_{\exists}^{\col,2} )=\bfrak(\Rbf_{\exists}^{\col,k} )$.   
\end{corollary}

We now turn to the corresponding variants.

Denote by $[\omega]^*$ the family of infinite co-infinite subsets of $\omega$ and denote
\[\omega^\omega_{>k}=\{f\in\omega^{\uparrow\omega}\colon \forall n\; |[\,f(n),f(n+1))|> k\}.\] 
Note that $\omega^\omega_{>0}=\omega^{\uparrow\omega}$.

\begin{definition}
Let $k >0$.
  \begin{enumerate}[label=\rm(\alph*)]
    \item For $f \in \omega^\omega_{>k}$ and $X \in [\omega]^*$, we define 
\[f\;\sqsubset_\forall^{k}\;X\; \iff\;\forall^\infty n\;\bigl|\,[\,f(n),f(n+1))\cap X\bigr|\le k,\]
\[f\;\sqsubset_\exists^{k}\;X\; \iff\;\exists^\infty n\;\bigl|\,[\,f(n),f(n+1))\cap X\bigr|\le k.\]
   \item Define the relational system $\Lbf_\iota^k=\la\omega^\omega_{>k},[\omega]^*,\sqsubset_\iota^k\ra$ for $\iota\in\{\forall,\exists\}$. 
\end{enumerate}  
\end{definition}

The case $k=0$ is exceptional.

\begin{remark}
Notice that for $k=0$, $\dfrak(\Lbf_\forall^0)$ is not well-defined (or ``$\infty$'') because no $X \in [\omega]^*$ can satisfy the relation; and $\bfrak(\Lbf_\forall^0) = 1$ because the negation of the relation is always satisfied for any infinite set.   
\end{remark}

\begin{fact}\label{fac:basL}
\ 
    \begin{enumerate}[label=\rm(\arabic*)]
        \item\label{fac:basLc}  If $k<\ell$ then $\Lbf^\ell_\iota\leqT\Lbf^k_\iota$ for $\iota\in\{\forall,\exists\}$. 
         \item\label{fac:basLb} For each $k\in\omega$, we have $\Lbf^k_\exists\leqT\Lbf^k_\forall$.
         \item $\Rbf^k_\iota\leqT\Lbf^k_\iota$ for $\iota\in\{\forall,\exists\}$.
    \end{enumerate}
\end{fact}

The next two lemmas show that $\Lbf^k_\forall$ is Tukey equivalent to $\Dbf$.

\begin{lemma}\label{lem:forleq}
For any $k >0$, we have $\Lbf_\forall^k\leqT\mathbf{D}$.     
\end{lemma}

\begin{proof}
 We explicitly define maps $\Psi_-\colon \omega^\omega_{>k} \to \omega^{\uparrow\omega}$ and $\Psi_+\colon \omega^{\uparrow\omega} \to [\omega]^*$ as follows. For $f\in\omega^{\uparrow\omega}$ define $f'\in\omega^{\uparrow\omega}$ such that for all $n\in\omega$, \[f(f'(n))+1<f'(n+ 1).\]
Then put $\Psi_+(f)=\ran(f')\in[\omega]^*$. On the other hand, for $h\in \omega^\omega_{>k}$ put $\Psi_-(h)(n)=h(n+1)$.

It remains to prove that if $\Psi_-(h)\leq^* f$, then $h\sqsubset^k_\forall \Psi_+(f)$. Assume $\Psi_-(h) \leq^* f$. Then there exists $N_0$ such that for all $n \ge N_0$, $h(n+1) \le f(n).$ Because $f'$ is strictly increasing and tends to infinity, we can choose $N_1$ large enough so that $f'(N_1) \ge N_0$. Then for all $j \ge N_1$,

\[
h(f'(j)+1) \le f(f'(j)).
\]
Using the defining property of $f'$, i.e., $f(f'(j)) + 1 < f'(j+1)$, we obtain
\[
h(f'(j)+1) \le f(f'(j)) < f'(j+1) - 1 < f'(j+1).
\tag{1}
\]
Thus, for all $j \ge N_1$,
\[
h(f'(j)+1) < f'(j+1).
\tag{2}
\]
Let $X = \Psi_+(f) = \operatorname{ran}(f')$. Consider an interval $I_n = [h(n), h(n+1))$ with $n$ large enough so that  $n>f'(N_1)+1$. This ensures that if $h(n) \le f'(j)$ for some $j$, then $j \ge N_1$ (since $h$, $f'$ is are strictly increasing and $n$ is large, so if $j<N_1$, then $f'(j)<f'(N_1)<h(f'(N_1))<h(f'(N_1)+1)<h(n)$).

We claim that for such large $n$, $|I_n \cap X| \le 1$. Indeed, suppose towards a contradiction that there exist two distinct indices $j, m \ge N_1$ with $j < m$ such that both $f'(j)$ and $f'(m)$ belong to $I_n$.  
Since $j \ge N_1$, inequality (2) holds: $h(f'(j)+1) < f'(j+1)$.  
Because $f'(m) \ge f'(j+1)$ (since $m > j$ and $f'$ is strictly increasing), we have
\[
f'(m) \ge f'(j+1) > h(f'(j)+1).
\]
Now, since both $f'(j)$ and $f'(m)$ are in $I_n = [h(n), h(n+1))$, we have $h(n) \le f'(j) < h(n+1)$. Because $h$ is strictly increasing, $f'(j) \ge h(n)$ implies $f'(j)+1 \ge n+1$, and therefore
\[
h(f'(j)+1) \ge h(n+1).
\]
But this contradicts the fact that $f'(m) < h(n+1)$ (since $f'(m) \in I_n$) together with $f'(m) > h(f'(j)+1)$.  
Hence, it is impossible for two distinct elements of $\operatorname{ran}(f')$ to lie in the same interval $I_n$ for large $n$.
\end{proof}

 \begin{lemma}\label{lem:forgeq}
 For any $k >0$, we have $(\Lbf_\forall^k)^\perp\leqT\mathbf{D}^\perp$.     
\end{lemma}

\begin{proof}
 We explicitly define maps $\Psi_+\colon \omega^{\uparrow\omega}\to\omega^\omega_{>k} $ and $\Psi_-\colon [\omega]^*\to\omega^{\uparrow\omega}$ as follows.  For $X\in[\omega]^*$ define 
\[f_X(n)=\min\{m>n:|[n,m)\cap X|>2k\},\]
so we define $\Psi_-(X) = g_X$, where $g_X$ is a strictly increasing function that majorizes $f_X$.
A standard way to ensure a strict increase is to define $g_X$ recursively:
\[g_X(0) = f_X(0),\]
\[g_X(n+1) = \max\{f_X(n+1), g_X(n) + 1\}.\]
Since $g_X(n+1) \geq g_X(n) + 1$ for any $n\in\omega$ we have $g_X \in \omega^{\uparrow\omega}$.
Also, $f_X(n) \leq g_X(n)$ for all $n$. Now, for $h\in\omega^{\uparrow\omega}$ define 
\[h'(0)=0,  \qquad h'(n+1)=k+1+h'(n)+h(h'(n)).\]
Note that $|h'(n+1)-h'(n)|=k+1+h(h'(n))>k$, thereby, $h'\in\Dbf_k$. So put $\Psi_+=h'$. Lastly, we prove that if $g_X(\leq^*)^\perp h$ then $X\;(\sqsubset_\forall^{k})^\perp\;h'$. Suppose that $g_X(\leq^* )^\perp h$, i.e., $h\not\leq^* g_X$. 

So assume that $m\in [h'(n), h'(n+1))$ such that $h(m)>g_X(m)$. Since $g_X(m) \geq f_X(m)$, $h(m)>f_X(m)$. On the other hand, since $h$ is strictly increasing, we have $h(m)<h(h'(n+1))$, which implies that $f_X(m)<h(h'(n+1))<h'(n+2)$ (recall $h'(n+2)=k+1+h'(n+1)+h(h'(n+1))$, so $h'(n+2)>h(h'(n+1))$).  

Hence, $[m,h'(n+2))\cap X|>2k$, so $[h'(n), h'(n+1))\cap X|>k$ or $[[h'(n+1), h'(n+2))\cap X|>k$. Consequently, $h'\;\not\sqsubset_\forall^{k}\;X$.
\end{proof}

Combining \autoref{lem:forleq} and \autoref{lem:forgeq}, we obtain the following result, which establishes the first part of \autoref{mthmfive}.
 
\begin{corollary}\label{a80}
For any fixed $k >0$, we obtain $\bfrak(\Lbf_\forall^k) = \bfrak$ and 
     $\dfrak(\Lbf_\forall^k) = \dfrak$.
\end{corollary}

Arguing as in \autoref{a60}, we obtain the following reduction.

\begin{lemma}\label{lem:RLex0<1}
$\Lbf^0_\exists \leqT\Lbf^1_\exists$.
\end{lemma}

We next compare $\Lbf_\exists^1$ with the relation $\Rbf_\exists^1$.

\begin{lemma}\label{lem:RsL}
    $\Lbf_\exists^1\eqT \Rbf_\exists^1$.
\end{lemma}
\begin{proof}
The Tukey connections are straightforward. To see $\leqT$, define $\Psi_-\colon \Dbf_1\to \omega^{\uparrow\omega}$ by $\Psi_-(f)=f$ and $\Psi_+\colon [\omega]^\omega\to [\omega]^*$ by $\Psi_+(X)=X$ if $X\in [\omega]^*$, otherwise $\Psi_+(X)$ is defined to be any set from $[\omega]^*$. Assume $f\in \Dbf_1$, $X\in [\omega]^*$ are such that $\Psi_-(f)=f\sqsubset_\exists^1 X$. i.e., 
$\exists ^\infty n\;  |X\cap [f(n), f(n+1))|\leq 1$. By the definition of $\Dbf_1$, all of the intervals $[f(n), f(n+1))$ contain $\geq 2$ points, so $X\in [\omega]^*$. Thus, $\Psi_+(X)=X$ and we have $f\sqsubset_\exists^1 \Psi_+(X)$.

To prove $\succeq_\mathrm{T}$, for any $f\in \omega^{\uparrow\omega}$ define $\Psi_-(f)(n)=f(2n)$ for all $n$. Clearly, $\Psi_-(f)\in \Dbf_1$ for all $f\in\omega^{\uparrow\omega}$. Put $\Psi_+(X)=X$ for all $X\in \Dbf_1$. Assume $\Psi_-(f)\sqsubset_\exists^1 X$, i.e., $\exists^\infty n\; |[f(2n), f(2(n+1))) \cap X|\leq 1$. Then we have $[f(2n), f(2n+1))\cap X=\emptyset$ or $[f(2n+1), f(2n+2))\cap X=\emptyset$. Thus, $f\sqsubset_\exists^1 X=\Psi_+(X)$.
\end{proof}

Combining \autoref{fac:basL}\ref{fac:basLc}, \autoref{lem:RLex0<1}, and \autoref{lem:RsL}, we obtain the following consequence.

\begin{corollary}\label{lem:f1}
    $\dfrak(\Lbf_\exists^0)=\bfrak$ and $\bfrak(\Lbf_\exists^0)=\dfrak$.
\end{corollary}

\autoref{lem:f1} together with \autoref{a32} and the inequalities $\bfrak(\Lbf_\exists^0) \le \bfrak(\Lbf_\exists^k)$ and  $\dfrak(\Lbf_\exists^k) \le \dfrak(\Lbf_\exists^0)$ yields the following.

\begin{corollary}\label{a81}
For any  $k>0$, we have $\dfrak(\Lbf_\exists^k)=\bfrak$ and $\bfrak(\Lbf_\exists^k)=\dfrak$.
\end{corollary}

The previous corollary completes the proof of \autoref{mthmfive}. We now turn to a measure-theoretic analogue of the preceding relational systems.

\begin{definition}\label{a130}
Let $\Pwf_\infty(\mathbb{R}_{\geq 0})$ denote the family of all sets with infinite Lebesgue measure on $\mathbb{R}_{\geq 0}$ and let $\varepsilon>0$.

    \begin{enumerate}[label=\rm(\alph*)]
        \item For any $f\in\omega^{\uparrow\omega}$ and $Y\in \Pwf_\infty(\mathbb{R}_{\geq 0})$ define \[f\vartriangleleft^\varepsilon_\forall Y\;\iff\; \forall^\infty n\;\mu([f(n), f(n+1))\cap Y)\leq \varepsilon,\]
        \[f\vartriangleleft^\varepsilon_\exists Y\;\iff\; \exists^\infty n\;\mu([f(n), f(n+1))\cap Y)\leq \varepsilon.\]
        \item Consider the relational system $\Mbf^\varepsilon_\iota=\la\omega^{\uparrow\omega},\Pwf_\infty(\mathbb{R}_{\geq 0}),\vartriangleleft^\varepsilon_\iota\ra$ for $\iota\in\{\forall.\exists\}$.
    \end{enumerate}
\end{definition}

We next show that the measure-theoretic systems introduced above are, from the Tukey viewpoint, equivalent to the discrete interval systems studied previously.

\begin{lemma}\label{a106}
For any $k>0$, we have $\Mbf^k_\forall \eqT \Rbf^k_\forall$.
\end{lemma}
\begin{proof}
$\leqT$: Define $\Psi_-=\id_{\omega^{\uparrow\omega}}$ and $\Psi_+\colon [\omega]^\omega\to \Pwf_\infty(\mathbb{R}_{\geq 0})$ by $\Psi_+(X)=\bigcup_{j\in X}[j, j+1)$.
Assume $f\sqsubset ^k_\forall X$ and let $N\in\omega$ be such that for all $n\geq N$, $|[f(n), f(n+1))\cap X|\leq k$. Fix $n\geq N$. Then \[[f(n), f(n+1))\cap \Psi_+(X)=\bigcup_{j\in [f(n), f(n+1))\cap X}[j, j+1).\]
Since $|[f(n), f(n+1))\cap X|\leq k$, we have 
\[\mu([f(n), f(n+1))\cap \Psi_+(X))\leq k\cdot 1=k.\]

$\succeq_\mathrm{T}$: Define $\Psi_-=\id_{\omega^{\uparrow\omega}}$ and $\Psi_+\colon \Pwf_\infty(\mathbb{R}_{\geq 0})\to [\omega]^\omega$ inductively as follows: for any $Y\in \Pwf_\infty(\mathbb{R}_{\geq 0})$ put $y_0=0$ and for $n\geq 1$,
\[y_j:=\min\set{m\in\omega}{\mu([y_{j-1},m)\cap Y)\geq 2}.\]
Then set $\Psi_+(Y)=\set{y_j}{j\in\omega}$.

Assume $f\vartriangleleft^k_\forall Y$ and let $N\in\omega$ be such that for all $n\geq N$ we have $\mu([f(n), f(n+1))\cap Y)\leq k$. Fix $n\geq N$. Assume, for sake of contradiction, that $|[f(n), f(n+1))\cap \Psi_+(Y)|\geq k+1$, i.e., there are at least $k+1$ points $y_{j_0}< \dotso< y_{j_k}$ of $\Psi_+(Y)$ inside $[f(n), f(n+1))$. By the definition of $\Psi_+$, we have $\mu(Y\cap [y_{j_0}, y_{j_1})) \geq 2$,  $\mu(Y\cap [y_{j_1}, y_{j_2})) \geq 2$, etc. Since each of these intervals is a subset of $[f(n), f(n+1))$ and the intervals are disjoint, we have $\mu(Y\cap[f(n), f(n+1)))\geq k\cdot 2>k$, which is a contradiction.
\end{proof} 

Proceeding as in \autoref{a106}, we arrive at the following reduction

\begin{lemma}
    $\Mbf^k_\exists \eqT \Rbf^k_\exists$.
\end{lemma}

The next lemma shows that the specific value of the parameter $\varepsilon$ is irrelevant up to Tukey equivalence.

\begin{lemma}\label{a110}
 For any 
$\varepsilon, \delta > 0$,  $\Mbf^\varepsilon_\forall\eqT \Mbf^\delta_\forall$.
\end{lemma}
\begin{proof}
Without loss of generality, assume $\varepsilon<\delta$. $\succeq_\mathrm{T}$ is trivial, so we need to prove only the $\leqT$ direction.  Let $B\in\omega$ be such that $\varepsilon\cdot B>\delta$. To define $\Psi_-\colon\omega^{\uparrow\omega}\to\omega^{\uparrow\omega}$, set $\Psi_-(f)(n)=B\cdot f(n)$ for every $n$.
To define $\Psi_+\colon \Pwf_\infty(\mathbb{R}_{\geq 0})\to\Pwf_\infty(\mathbb{R}_{\geq 0})$, first define a contraction $c\colon \mathbb{R}_{\geq 0}\to\mathbb{R}_{\geq 0}$ by $c(x)=\frac{x}{B}$ for every $x\in\mathbb{R}_{\geq 0}$. Now put $\Psi_+(Y)=c[Y]$.

Assume $B\cdot f\vartriangleleft^\delta_\forall Y$, i.e., there is $N\in\omega$ such that for every $n\geq N$ 
\[\mu([B\cdot f(n), B\cdot f(n+1))\cap Y)\leq \delta.\]
Fix $n\geq N$ and denote 
\[A_{f,Y}=[B\cdot f(n), B\cdot f(n+1))\cap Y.\]
Since the function $c$ is just a linear contraction, we can use the scaling property of Lebesgue measure (i.e., $\mu(\alpha X)=|\alpha|\mu(X)$ for nonzero $\alpha$):
\[\mu(c[A_{f,Y}])=\frac{1}{B}\mu(A_{f,Y})\leq \frac{\delta}{B}<\varepsilon.\]
Notice that by the injectivity of $c$ we have
\[c[[B\cdot f(n), B\cdot f(n+1))\cap Y]= [f(n), f(n+1))\cap c[Y],\]
so
$\mu([f(n), f(n+1))\cap c[Y])\leq \varepsilon$. Since this holds for any $n\geq N$, we get $f\vartriangleleft^\varepsilon_\forall c[Y]= \Psi_+(Y)$.
\end{proof}

The following reduction follows by the same argument as in \autoref{a110}.

\begin{lemma}
    For any 
$\varepsilon, \delta > 0$, 
$\Mbf^\varepsilon_\exists \eqT \Mbf^\delta_\exists$.
\end{lemma}

Combining the previous two lemmas, we obtain the following characterization.

\begin{corollary}
For any $\varepsilon>0$ and $k>0$ we have $\Mbf^\varepsilon_\forall \eqT \Rbf^k_\forall$.
\end{corollary}

\begin{corollary}\label{a83}
$\bfrak(\Mbf^\varepsilon_\forall)=\bfrak$ and  $\dfrak(\Mbf^\varepsilon_\forall)=\dfrak$.
\end{corollary}

Combining the $\exists$-analogues of the previous two lemmas yields the following characterization. 

\begin{corollary}
For any $\varepsilon>0$ and $k>0$ we have
$\Mbf^\varepsilon_\exists \eqT \Rbf^k_\exists$.
\end{corollary}

\begin{corollary}\label{a116}
$\bfrak(\Mbf^\varepsilon_\exists)=\bfrak$
and
$\dfrak(\Mbf^\varepsilon_\exists)=\dfrak$.
\end{corollary}

As an immediate consequence of the previous corollaries, \autoref{mthmsix} follows. 

Finally, we consider a variant in which the admissible number of intersections increases with the interval index.

\begin{definition}
 Let $\omega^{>\id}=\set{f\in \omega^{\uparrow\omega}}{f(n+1)-f(n)>n}$.
    \begin{enumerate}[label=\rm(\alph*)]
        \item For any $f\in\omega^{>\id}$ and $X\in [\omega]^\omega$ let 
\[f\sqsubset^\id_\forall X\; \iff\; \forall^\infty n\; \bigl|[f(n), f(n+1))\cap X\bigr|\leq n,\]
\[f\sqsubset^\id_\exists X\; \iff\; \exists^\infty n\; \bigl|[f(n), f(n+1))\cap X\bigr|\leq n.\]
        \item Define the relational system $\Rbf^\id_\iota=\la \omega^{>\id}, [\omega]^\omega, \sqsubset^\id_\iota \ra$ for $\iota\in\{\forall,\exists\}$.
    \end{enumerate}
\end{definition}

The next lemma relates this variable-threshold relation to the systems previously considered and to the classical dominating relation.

\begin{lemma}\label{a120}
$\Rbf^\id_\forall \leqT \Rbf^k_\forall$ and $\Dbf\leqT \Rbf^\id_\forall$. 
\end{lemma}
\begin{proof}
  The first Tukey connection is immediate, since both maps are identities. Now define $\Dbf^\id=\la \omega^{>\id}, \omega^{>\id}, \leq^* \ra$. It is easy to see that $\Dbf^\id\eqT \Dbf$, so we can work with $\Dbf^\id$ instead of $\Dbf$. Define $\Psi_-=\id_{\omega^{>\id}}$ and $\Psi_+\colon [\omega]^\omega\to \omega^{>\id}$ by $\Psi_+(X)=h$, where $h\in\omega^{>\id}$ is any function such that $\bar{x}:=\seq{x_{n^2}}{n\in\omega}\leq^* h$, i.e.,  $\forall^\infty n\; x_{n^2}\leq h(n)$. Assume $N\in\omega$ is such that $\forall n\geq N\; |[f(n), f(n+1))\cap X|\leq n$ and denote $C=|[0, f(N))\cap X|$. Note that for any $n\geq N$, there are at most $C+ (n-N)(n-1)$ points of $X$ inside $[0, f(n))$. Fix $n'>\max\{N, C\}$. Since $C+ (n'-N)(n'-1)\leq C+(n')^2-n',$ the largest possible point from $X$ inside $[0, f(n'))$  is $x_{C+(n')^2-n'}$. Since $x_{C+(n')^2-n'}<x_{(n')^2}$, necessarily $f(n')\leq x_{(n')^2}=\bar{x}(n')$. Therefore, $f\leq^* \bar{x}\leq^* h=\Psi_+(X)$.
\end{proof}

\begin{corollary}
  $\dfrak(\Rbf^\id)=\dfrak$ and $\bfrak(\Rbf^\id)=\bfrak$.   
\end{corollary}

 More generally, for any non-decreasing function $g\in\omega^\omega$, one may define the corresponding notions $\omega^{>g}$, $\sqsubset^g_\forall$, and $\Rbf^g_\forall$. The same argument yields $\bfrak(\Rbf_\forall^g)=\bfrak$ and $\dfrak(\Rbf_\forall^g)=\dfrak$.
 
\begin{lemma}\label{a121}
    $\Rbf^\id_\exists\eqT \Rbf^0_\exists$.
\end{lemma}
\begin{proof}
    $\Rbf^\id_\exists\leqT \Rbf^0_\exists$ is clear, so we show $\Rbf^0_\exists\leqT\Rbf^\id_\exists$. For any $f\in\omega^{\uparrow\omega}$ define $h$ recursively by $h(0)=0$, for each $n\geq 0$
    \[h(n+1)=f(h(n)+n+1),\]
    and put $\Psi_-(f)=h$. Set $\Psi_+=\id_{[\omega]^\omega}$.

    Assume $|[h(n), h(n+1))\cap X|\leq n$ for infinitely many $n\in\omega$ and fix such an $n$.
    Since $f$ is increasing, we have $h(n)<f(h(n))<f(h(n)+n+1)=h(n+1)$. Note that there are at least $n+1$ intervals constituted by $f$ inside $[h(n), h(n+1))$, namely \[[f(h(n)), f(h(n)+1)), [f(h(n)+1), f(h(n)+2)), \dotso, [f(h(n)+n)), f(h(n)+n+1)),\]
    thus, at least one of these intervals has an empty intersection with $X$.
\end{proof}

\begin{corollary}
$\dfrak(\Rbf^\id_\exists)=\bfrak$ and $\bfrak(\Rbf^\id_\exists)=\dfrak$.
\end{corollary}

Similarly, for the relation 
$\sqsubset_\exists^g$, it can be obtained $\bfrak(\Rbf_\exists^g)=\bfrak$ and $\dfrak(\Rbf_\exists^g)=\dfrak$.

\begin{definition}
Let $\omega^\omega_\uparrow = \set{f\in \omega^\omega}{f(n+1)-f(n)\to\infty}$.  
    \begin{enumerate}[label=\rm(\arabic*)]
        \item For $f \in \omega^\omega_\uparrow$ and $X \in [\omega]^\omega$ let \[f\;\sqsubset_\forall^{\mathrm{bd}}\;X\quad\Longleftrightarrow\quad\exists k\in\omega\;\forall^\infty n\;\bigl|\,[\,f(n),f(n+1))\cap X\bigr|\le k,\]
\[f\;\sqsubset_\exists^{\mathrm{bd}}\;X\quad\Longleftrightarrow\quad\exists k\in\omega\;\exists^\infty n\;\bigl|\,[\,f(n),f(n+1))\cap X\bigr|\le k.\]
        \item Define the relational system $\Rbf_\iota^{\mathrm{bd}}=\la\omega^\omega_\uparrow,[\omega]^\omega,\sqsubset_\iota^{\mathrm{bd}}\ra$ for $\iota\in\{\forall,\exists\}$. 
    \end{enumerate}
\end{definition}

The bounded versions are naturally below the fixed-threshold systems in the Tukey order.

\begin{fact}\label{varIVa}
$\Rbf_\iota^{\mathrm{bd}}\leqT\Rbf_\iota^k$ for $\iota\in\{\forall,\exists\}$. In particular, $\dfrak(\Rbf_\iota^{\mathrm{bd}})\leq\dfrak(\Rbf_\iota^k)$ and $\bfrak(\Rbf_\iota^k)\leq \bfrak(\Rbf_\iota^{\mathrm{bd}})$.  
\end{fact}

We next show that, in the existential case, the bounded relation is already strong enough to recover the behavior of $\Rbf_\exists^k$.

\begin{lemma}\label{varIVb}
    $\Rbf^k_\exists\leqT\Rbf^\mathrm{bd}_\exists$ for $k\geq 0$, in particular $\bfrak\leq \dfrak(\Rbf^\mathrm{bd}_\exists)$ and $\bfrak(\Rbf^\mathrm{bd}_\exists)\leq \dfrak$.
\end{lemma}
\begin{proof}
Define $\Psi_-(f)(n)=f(n^2)$ for every $n$. Observe that each interval \[[\Psi_-(f)(n), \Psi_-(f)(n+1))\] contains $(n+1)^2-n^2=2n+1$ many intervals $[f(j), f(j+1))$. Let $\Psi_+=\id_{[\omega]^\omega}$.

    Assume $\exists k\;\exists^\infty n\; |[f(n^2), f(n^2+2n+1))\cap X|\leq k$. Let $N\in\omega$ be such that $2N+1>k$. Then for every $n\geq N$, the interval $[f(n^2), f(n^2+2n+1))$ contains $\geq 2N+1>k$ intervals $[f(j), f(j+1))$, so at least one of them must have empty intersection with $X$, i.e., $f\sqsubset^0_\exists X$ and consequently $f\sqsubset^k_\exists X$. 
\end{proof}

Combining \autoref{varIVa}, \autoref{varIVb}, and \autoref{a32}, we obtain the following result, which establishes the second part of \autoref{mthmtwo}.

\begin{corollary}\label{a84}
   $\dfrak(\Rbf^\mathrm{bd}_\exists)=\bfrak$ and $\bfrak(\Rbf^\mathrm{bd}_\exists)= \dfrak$.
\end{corollary}

\begin{lemma}\label{a100}
$\Dbf\leqT \Rbf^{\mathrm{bd}}_\forall$, in particular
 $\bfrak(\Rbf_\forall^{\mathrm{bd}})\leq\bfrak$   and $\dfrak\leq\dfrak(\Rbf_\forall^{\mathrm{bd}})$.
\end{lemma}

\begin{proof}
We define the maps $\Phi_-$ and $\Phi_+$ as follows. For any $g \in \omega^{\uparrow\omega}$, let $f = \Phi_-(g)$ be defined recursively:
\begin{itemize}
    \item  $f(0) = g(0)$,
    \item  $f(n+1) = g(f(n)) + n + 1$.
\end{itemize}

Since $g$ is strictly increasing, $g(f(n)) \ge f(n) + 1$. Thus, $f(n+1) - f(n) = g(f(n)) - f(n) + n + 1 \ge n + 2$. As $n \to \infty$, $(f(n+1) - f(n)) \to \infty$, so $f \in\omega^\omega_\uparrow$. Additionally, $f(n+1) > f(n)$, so $f$ is strictly increasing and $f(n) \ge n$ for all $n$.

For any $X \in [\omega]^\omega$, let $X = \{x_0, x_1, x_2, \dots\}$ be its increasing enumeration. We define $h = \Phi_+(X)$ by
\[ h(n) = x_{n^2}.\]

We must show that if $\Phi_-(g) \sqsubset_\forall^{\mathrm{bd}} X$, then $g \leq^* \Phi_+(X)$.

Assume $f \sqsubset_\forall^{\mathrm{bd}} X$. By the definition, there exist $k \in \omega$ and $N \in \omega$ such that for all $i \ge N$,
\[| [f(i), f(i+1)) \cap X | \le k.\]

For each $n\in\omega$, let $m_n$ be the index such that $x_{m_n}$ is the smallest element of $X$ with $x_{m_n} \ge f(n)$. For any $n > N$, the elements of $X$ below $f(n)$ consist of those below $f(N)$ and those in the intervals $[f(i), f(i+1))$ for $i = N, \dots, n-1$. Therefore
\[m_n \le |X \cap [0, f(N))| + k \cdot (n - N).\]
This shows that $m_n$ is bounded by a linear function of $n$, say $m_n \le A \cdot n + B$.

By the definition of $f$, we have $f(n+1) > g(f(n))$. Since $| [f(n), f(n+1)) \cap X | \le k$, the subinterval $[f(n), g(f(n)))$ contains at most $k$ points of $X$. 

The point $x_{m_n}$ is the first point of $X$ such that $x_{m_n} \ge f(n)$. Among the set of $k+1$ consecutive points $\{x_{m_n}, x_{m_n+1}, \dots, x_{m_n+k}\}$, at least one must fall outside the interval $[f(n), g(f(n)))$. Therefore, the largest of these points must satisfy
\[x_{m_n+k} \ge g(f(n)).\]

Since $f(n) \ge n$ and $g$ is strictly increasing, we have
\[g(n) \le g(f(n)) \le x_{m_n+k}.\]

Since $m_n$ grows linearly ($m_n \le An+B$), there exists an $N_0$ such that for all $n \ge N_0$, $n^2 \ge m_n + k$. Thus,
\[g(n) \le x_{n^2} = h(n).\]
This implies $g \leq^* h$, completing the Tukey reduction $\mathbf{D} \leqT \Rbf_\forall^{\mathrm{bd}}$   
\end{proof}

Combining \autoref{varIVa} together with \autoref{mthmone} and \autoref{a100}, we obtain the following result

\begin{corollary}\label{a101}
 $\dfrak(\Rbf_\forall^{\mathrm{bd}})=\dfrak$ and $\bfrak=\bfrak(\Rbf_\forall^{\mathrm{bd}})$.   
\end{corollary}

This completes \autoref{mthmtwo}.

\section{Open questions}

Motivated by \autoref{a130}, it is natural to consider stronger asymptotic forms of interval avoidance and to ask whether they still recover the classical bounding and dominating numbers. In this direction, we introduce the following variants of the relation $\vartriangleleft^\varepsilon_\forall$.

\begin{enumerate}
    \item Let 
$\vec{\varepsilon} = \langle \varepsilon_n : n < \omega \rangle$ be a sequence of positive real numbers such that 
$\lim_{n \to \infty} \varepsilon_n = 0$. We define
\[f \vartriangleleft^{\vec{\varepsilon}}_\forall Y\; \iff\; \forall^\infty n \in \omega \left( \mu([f(n), f(n+1)) \cap Y) \leq \varepsilon_n \right).\]
    \item We say that 
$f$ strongly captures 
$Y$ if the total intersection across all intervals is finite
\[f \vartriangleleft^{sum} Y\; \iff\; \sum_{n=0}^{\infty} \mu([f(n), f(n+1)) \cap Y) < \infty.\]
\end{enumerate}
These relations form a natural hierarchy of increasing strength
\[f \vartriangleleft^{sum} Y \implies f \vartriangleleft^{\vec{\varepsilon}}_\forall Y \implies f \vartriangleleft^\varepsilon_\forall Y.\]
Consequently, the associated cardinal invariants satisfy $\mathfrak{b}(\Mbf^\varepsilon_\forall)\leq\mathfrak{b}(\Mbf^{sum})$ and $\mathfrak{d}(\Mbf^{sum})\leq\mathfrak{d}(\Mbf^\varepsilon_\forall)$.

This naturally leads to the following problem, which asks whether the summable strengthening still yields the classical characteristics.

\begin{question}
Is $\mathfrak{b}(\Mbf^{sum}) \leq \mathfrak{b}$ and $\mathfrak{d}\leq\mathfrak{d}(\Mbf^{sum})$ provable in $\thzfc$?
\end{question}

%\printbibliography

{\small
\bibliography{bd}
\bibliographystyle{alpha}
}

%\Addresses

\end{document}